\documentclass[12pt]{amsart}

\usepackage{amsrefs}
\usepackage{amssymb}
\usepackage{fancyhdr}
\usepackage{bbm}
\usepackage{xcolor}
\usepackage{hyperref}
\usepackage{titling}

\hypersetup{
  colorlinks=true,
  linkcolor=blue,
  citecolor=blue
}
\usepackage{mleftright}
\usepackage{amsaddr}

\calclayout

\newtheorem{theorem}{Theorem}
\newtheorem*{theorem*}{Theorem}
\newtheorem{lemma}[theorem]{Lemma}
\newtheorem{proposition}[theorem]{Proposition}
\newtheorem{claim}[theorem]{Claim}
\newtheorem{corollary}[theorem]{Corollary}

\newtheorem{maintheorem}{Theorem}

\theoremstyle{definition}

\newtheorem*{definition*}{Definition}
\newtheorem*{lemma*}{Lemma}
\usepackage{graphicx}
\usepackage{pstricks, enumerate, pst-node, pst-text, pst-plot}

\numberwithin{equation}{section}
\numberwithin{theorem}{section}

\newcommand{\R}{\mathbb{R}}

\newcommand{\Z}{\mathbb{Z}}

\usepackage{xparse}
\DeclareDocumentCommand\Pr{ m g }{\ensuremath{
    {   \IfNoValueTF {#2}
      {\mathbb{P}\mleft[{#1}\mright]}
      {\mathbb{P}\mleft[{#1}\middle\vert{#2}\mright]}%
    }
}}
\DeclareDocumentCommand\E{ m g }{\ensuremath{
    {   \IfNoValueTF {#2}
      {\mathbb{E}\mleft[{#1}\mright]}
      {\mathbb{E}\mleft[{#1}\middle\vert{#2}\mright]}%
    }
}}

\def\dd{\mathrm{d}}
\def\ee{\mathrm{e}}
\def\cP{\mathcal{P}}
\def\cPr{\mathrm{Prob}}
\def\cM{\mathcal{M}}
\def\F{\mathcal{F}}

\def\Q{\mathbb{Q}}
\def\C{\mathbb{C}}

\begin{document}

\title{On the origin of the Boltzmann Distribution}

\author{Fedor Sandomirskiy}
\address{Princeton University}
\author{Omer Tamuz}
\address{California Institute of Technology}

\thanks{Omer Tamuz was supported by a BSF award (\#2018397) and a National Science Foundation CAREER award (DMS-1944153).}

\date{\today}

\maketitle

\begin{abstract}

The family of Boltzmann distributions is used in statistical mechanics to describe the distribution of states in systems with a given temperature. We give a novel characterization of this family  
as the unique one satisfying independence for uncoupled systems. The theorem boils down to a statement about endomorphisms of the convolution semi-group of finitely supported probability measures on the natural numbers, or, alternatively, about endomorphisms of the multiplicative semi-group of polynomials with non-negative coefficients.
\end{abstract}

\section{Introduction}

In statistical mechanics, a system can be found in various possible states. A state refers to a complete microscopic description of the system which encodes all physically relevant variables, and in particular determines the system's energy.\footnote{Statistical mechanics serves only as a motivation for the mathematical question we study. For a brief introduction to the relevant physics concepts, see Section~2.1 of \cite{mezard2009information}. An uninterested reader may proceed directly to the ``Definitions and Results'' section.}
In an equilibrium with a given temperature, the distribution of states is given by the Boltzmann distribution: The probability that the system has energy $E$ is proportional to the number of states with energy $E$ times $\ee^{-\beta E}$, where $1/\beta$ is proportional to the temperature of the environment.
Formally, if $\mu$ is a probability measure on the reals describing the fraction of states with given energy---i.e., for measurable $A \subseteq \R$, the fraction of states with energy in $A$ is $\mu(A)$---then the distribution of energy is given by the measure $\Phi_\beta[\mu]$ where
\begin{align*}
    \dd\Phi_\beta[\mu](E) = C \ee^{-\beta E}\dd \mu(E),
\end{align*}
and $C$ is the normalization constant.

The map $\Phi_\beta$ that assigns to each probability measure $\mu$ on $\R$ the measure $\Phi_\beta[\mu]$ has two important properties:
First, it preserves the measure class of $\mu$, so that $\mu$ and $\Phi_\beta[\mu]$ are mutually absolutely continuous.
Second, it commutes with convolution:
\begin{align}
\label{eq:commute}
 \Phi_\beta[\mu_1 * \mu_2] = \Phi_\beta[\mu_1] * \Phi_\beta[\mu_2].
\end{align}
In terms of the physics, the first property means that the system can only be found at an energy level that corresponds to an existing state, and conversely, any existing state can be attained.
The second property involves products of independent systems. Suppose that $\mu_1$ and $\mu_2$ describe the distribution of states in two systems. Form a new system whose set of states is the product of the two sets of states, and whose energy in each state is the sum of the two corresponding energies; this corresponds to no interaction between the two subsystems. Then the convolution $\mu_1 * \mu_2$ describes the new system, and \eqref{eq:commute} follows from the assumption that the joint distribution of states is the product measure of the distributions in the two subsystems, or, differently put, the assumption that  there is no correlation between systems that do not interact.

The usual  explanation for the Boltzmann distribution is one of maximum entropy~\cite{jaynes1957information}. The physics behind this stems from the idea that a system of interest is in contact with the environment---a large ``heat bath''---and that the state of the combined system is distributed uniformly over a fixed-energy surface in phase space.  

In this paper, we offer an alternative explanation: Our main result is that the members of the family $(\Phi_\beta)_{\beta \in \R}$ are the unique maps that
are measure-class-preserving 
and commute with convolution. Note that $\Phi_\beta$ is not well-defined for every probability measure $\mu$ on~$\R$, since normalization is impossible when the tails are too thick. We limit ourselves to 
finitely supported probability measures, and furthermore to measures  with support in either the natural numbers, integers or rationals.

The map $\Phi_\beta$ is also the tilting map, whose usefulness in the theory of large deviations stems from the fact that it commutes with convolutions; see, e.g.,~\cite[Lemma 2.6.4]{durrett2019probability}. It is thus natural to ask which maps from measures to measures are like tilting, in the sense that they preserve the measure class 
and commute with convolution. In this context, our main result is a negative one, stating that none other exist. 

\bigskip

\subsection*{Definitions and Results.} Given a subset $S \subseteq \R$ closed with respect to addition, denote by $\cPr(S)$ the set of finitely supported probability measures on $S$. This is a semi-group under the operation of convolution. We say that $\Phi \colon \cPr(S) \to \cPr(S)$ is \emph{support-preserving} if $\mu$ and $\Phi[\mu]$ have the same support, or, equivalently, are mutually absolutely continuous, for all $\mu \in \cPr(S)$; note that since these are finitely supported measures, the two notions indeed coincide. We say that $\Phi$ is an \emph{endomorphism} if it commutes with convolution, i.e., if $\Phi[\mu_1 * \mu_2] = \Phi[\mu_1] * \Phi[\mu_2]$. 

The map $\Phi_\beta \colon \cPr(S) \to \cPr(S)$ given by
\begin{align*}
    \Phi_\beta[\mu](s) = \frac{\mu(s)\ee^{-\beta s}}{\sum_t \mu(t)\ee^{-\beta t}}
\end{align*}
is easily verified to be a support-preserving endomorphism. Our main result is that these are the unique ones, when $S$ is $\Z_{\geq 0}$, $\Z$ or $\Q$. 

\begin{maintheorem}
\label{thm:main}
Suppose that $S$ is either $\Z_{\geq 0}$, $\Z$, or $\Q$. Then, for every  support-preserving endomorphism $\Phi$ of $\cPr(S)$, there exists a constant $\beta \in \R$ such that $\Phi = \Phi_\beta$.
\end{maintheorem}
For $S=\R$, the corresponding claim is not true, as we discuss below. However, it does hold if we also require $\Phi$ to be weakly continuous, as a corollary of the statement for $S=\Q$.

Note that even though $\Phi$ is not assumed to be a bijection, this property emerges as a consequence of the assumptions of Theorem~\ref{thm:main}. Another emergent property is that, up to normalization, $\Phi$ is affine, i.e., there is a function $f \colon \R \to \R$ such that $\Phi[\mu](s)$ is proportional to $\mu(s) f(s)$. We do not assume affinity; under such an additional assumption, it is easy to show that $f$ is an exponential.

The bulk of the effort in the proof of Theorem~\ref{thm:main} is the case $S=\Z_{\geq 0}$. By considering probability-generating functions, this question can be reduced to a question about polynomials. Denote by $\cP$ the polynomials in one variable whose coefficients are non-negative and sum to one:
\begin{align*}
    \cP = \left\{p(x) = \sum_{k=0}^n p_k x^k\,\middle\vert\, p_k \geq 0,\  p(1)=1\right\}.
\end{align*}
These are precisely the probability-generating functions of finitely supported probability measures on $\Z_{\geq 0}$. We note that representing probability measures by probability-generating functions is equivalent to looking at their Fourier transforms.  The representation of a probability measure by its probability-generating function carries a trade-off: while simplifying convolution of measures to mere multiplication, this representation makes it harder to handle the support-preserving property.   

We say that $\Phi \colon \cP \to \cP$ is support-preserving if it preserves the set of  positive coefficients. That is, if $p' = \Phi[p]$, then $p'_k > 0$ if and only if  $p_k>0$.  We say that $\Phi$ is multiplicative if $\Phi[p \cdot p'] = \Phi[p] \cdot \Phi[p']$. These two properties, translated back to probability measures, are equivalent to $\Phi$ being a support-preserving endomorphism.
\begin{maintheorem}
\label{thm:main-poly}
For every  support-preserving multiplicative $\Phi \colon \cP \to \cP$  there exists a constant  $\gamma > 0$ such that $\Phi[p](x) = p(\gamma x)/p(\gamma)$.
\end{maintheorem}
By identifying finitely supported probability measures on $\Z_{\geq 0}$  with their probability-generating functions, this result is equivalent to the case $S=\Z_{\geq 0}$ in Theorem~\ref{thm:main}.

A corollary of Theorem~\ref{thm:main-poly} is that if $p$ is a polynomial with at least two terms---e.g., $p(x)=(x+1)/2$---then $\Phi$ is completely determined by $\Phi[p]$: if $\Phi'[p]=\Phi[p]$ for two support-preserving multiplicative maps $\Phi,\Phi'$, then $\Phi=\Phi'$. 
\begin{corollary}\label{cor_unique}
    Suppose that $\Phi,\Phi' \colon \cP \to \cP$ are support-preserving and multiplicative and $p\in \cP$ is a polynomial with at least two terms. Then $\Phi[p]=\Phi'[p]$ implies that $\Phi=\Phi'$.
\end{corollary}
    Indeed, by Theorem~\ref{thm:main-poly} there are $\gamma,\gamma'>0$ such that $\Phi[p](x)=p(\gamma x)/p(\gamma)$ and $\Phi'[p](x)=p(\gamma' x)/p(\gamma')$. Comparing coefficients at the lowest-degree monomials, we conclude that  $\Phi[p]=\Phi'[p]$ implies $\gamma=\gamma'$ and hence $\Phi=\Phi'$.

This corollary may be a priori surprising since it is not clear how the multiplicative and support-preserving properties of $\Phi$ imply that  fixing $\Phi[(x+1)/2]$  constrains $\Phi[q]$ for any $q$ that is not a power of $(x+1)/2$. Indeed, one could have imagined that any choice of $\Phi[(x+1)/2]$ and, say, $\Phi[(x^2+x^{17})/2]$ that is support-preserving could be extended to a support-preserving and multiplicative $\Phi$. Nevertheless, this is generally impossible, as a consequence of Theorem~\ref{thm:main-poly}. The underlying reason is that $\cP$ is not a unique factorization domain. 

Say that $p \in \cP$ is irreducible if it cannot be written as a product $p=q_1\cdot q_2$, for $q_1, q_2 \in \cP$ such that $q_1,q_2 \neq 1$. It is easy to see that every $p \in \cP$ can be written as a product of irreducibles. But importantly,  this decomposition is not always unique. Hence, if $p = q_1 \cdot q_2 = r_1 \cdot r_2$, then $\Phi[q_1] \cdot \Phi[q_2] = \Phi[r_1] \cdot \Phi[r_2]$, providing additional constraints on $\Phi$. As it turns out, there are sufficiently many such constraints for $\Phi[(x+1)/2]$ to fix $\Phi$.

\subsection*{Proof techniques} 

The main tool in the proof of Theorem~\ref{thm:main-poly} is the extension of a support-preserving multiplicative $\Phi$ to the larger domain 
\begin{align*}
    \cM = \left\{p(x) = \sum_{k=0}^d p_k x^k\,\middle\vert\, p(x) > 0 \text{ for all } x>0,\ p(1)=1\right\}.
\end{align*}
These are the polynomials that are positive for positive $x$ and whose coefficients sum to $1$. Equivalently, these are the generating functions of signed, finitely supported, unit mass measures on $\Z_{\geq 0}$ that have a positive moment generating function.\footnote{In statistical mechanics terms, these measures describe systems that have \emph{anti-states}---a negative number of states at some energies---but still have a positive partition function. It is unclear if these have a meaningful physical interpretation.}

The fact that $\Phi$ can be extended to a multiplicative map on $\cM$ follows from the following classical result due to Poincar{\'e}~\cite{poincare1883equations}.
\begin{lemma}[Poincar{\'e}]
\label{lem:rs}
   For every $p \in \cM$ there is an $r \in \cP$ such that $p \cdot
   r \in \cP$.
\end{lemma}

Using this, we can extend the domain of any support-preserving multiplicative $\Phi$ to $\cM$ by choosing for  $p \in \cM$ an $r \in \cP$ such that $p\cdot r \in \cP$ and setting
\begin{align*}
    \Phi[p] = \frac{\Phi[p \cdot r]}{\Phi[r]}.
\end{align*}
As we show, this is well-defined, i.e., independent of the choice of $r$. However, $\Phi[p]$ could now be a rational function. The first part of our proof is dedicated to showing that the image of this extension of $\Phi$ is, in fact, in $\cM$. 

The advantage of $\cM$ is that it is a unique factorization domain: This set consists of the polynomials $p$ with $p(1)=1$ that have no positive roots, and hence, by the fundamental theorem of algebra, each $p \in \cM$ can be written as a product of linear terms with non-positive roots and quadratic terms without real roots, all in $\cM$, and this decomposition is unique. Since $\Phi$ commutes with multiplication, in order to show that $\Phi[p]=\Phi_\beta[p]$ for all $p \in \cM$, it suffices to show that this holds for linear and quadratic $p$. This is what the remainder of our effort is dedicated to.


\subsection*{Polynomials with rational coefficients}

Recall that, in the physics interpretation, a measure $\mu$ represents the fraction of states at each energy level.
In quantum mechanical settings, there are often only finitely many states at each energy level. In this case, the measure $\mu$ will have rational probabilities, and  the corresponding probability-generating function will be a polynomial with rational coefficients. Indeed, if there are $k_i$ states with energy $E_i$ and the total number of states is $k=\sum_i k_i$, then $\mu$ has an atom of weight $k_i/k$ at $E_i$.
This motivates the pursuit of the same questions, but in a rational setting. Let
\begin{align*}
    \cP_\Q = \left\{p(x) = \sum_{k=0}^d p_k x^k\,\middle\vert\, p_k \in \Q_{\geq 0}, \ \  p(1)=1\right\}
\end{align*}
be the set of polynomials with non-negative rational coefficients that sum to one. The next result shows that Theorem~\ref{thm:main-poly} still holds in this setting.

\begin{maintheorem}
\label{thm:main-poly-rational}
For every  support-preserving multiplicative $\Phi \colon \cP_\Q \to \cP$  there exists a constant  $\gamma >0$ such that $\Phi[p](x) = p(\gamma x)/p(\gamma)$.
\end{maintheorem}
If the image of $\Phi$ is further restricted to belong to $\cP_\Q$, then the resulting parameter $\gamma$ will also belong to $\Q$.

The proof of Theorem~\ref{thm:main-poly-rational} requires additional arguments beyond those of Theorem~\ref{thm:main-poly}.\footnote{The same proof technique shows that a version of Theorem~\ref{thm:main-poly-rational} holds for any subfield of $\R$ instead of $\Q$. We focus on the case of $\Q$ since other subfields do not result in a clear physical interpretation.} The main difficulty is that the rationality of the coefficients leads to more (and more complicated) irreducible polynomials. We circumvent this issue by using the arguments from the proof of Theorem~\ref{thm:main-poly} to show a similar statement for a dense sub-semi-group that does have simple irreducibles, and then proving the following automatic continuity type result.

We say that a set $P$ of polynomials is rich if $P$ contains $q(x)=x$ and, with each polynomial of the form $p(x)=x^m\cdot r(x)$ contained in $P$, the polynomial $r(x)$ is also contained in $P$. We endow $\cM$ with the topology of simultaneous convergence of the coefficients and the degree. 

\begin{proposition}
\label{prop:auto-cont}
Let $M\subset M'$ be rich dense sub-semi-groups of $\cM$.
Suppose $\Phi \colon M' \to \cM$ is multiplicative and
degree-preserving, $\Phi[p]\in \cP$ for $p\in M'\cap \cP$, 
and the restriction of $\Phi$  to  $M$ is the identity map. Then $\Phi$ is the identity map. 
\end{proposition}
Here, ``degree-preserving'' means that $q$ and $\Phi[q]$ are polynomials of the same degree.

\subsection*{Open questions}

\subsubsection*{Probability measures over the reals}

Endow the set $\cPr(\R)$ of finitely supported probability measures on $\R$ with the topology defined by $\mu_n \to \mu$ if $\int f \,\dd\mu_n \to \int f\,\dd\mu$ for all continuous $f \colon \R \to \R$. Then $\cPr(\Q)$ is a dense subset of $\cPr(\R)$, and so any continuous support-preserving endomorphism of $\cPr(\R)$ is of the form $\Phi_\beta$.

There are many other support-preserving endomorphisms of $\cPr(\R)$. Indeed, if $\pi \colon \R \to \R$ is any non-continuous solution to the Cauchy equation $\pi(x+y)=\pi(x)+\pi(y)$, then 
\begin{align*}
    \Phi[\mu](x) = \frac{\mu(x) \ee^{-\pi(x)}}{\sum_y \mu(y)\ee^{-\pi(y)}}
\end{align*}
is a support-preserving endomorphism. Of course, this is non-constructive, since the existence of non-continuous solutions of the Cauchy equation requires an application of some axiom of choice (e.g., any axiom that is strong enough to guarantee a Hamel  basis of $\R$ over $\Q$). A natural conjecture is that any support-preserving endomorphism on $\cPr(\R)$ that is not continuous (equivalently, not of the form $\Phi_\beta$) is not measurable.

The same conjecture can be made about support-preserving endomorphisms of the set of compactly supported (rather than finitely supported) probability measures on the reals. In this case, we suspect that the conjecture follows from automatic continuity results for Polish groups, such as the Banach-Pettis Theorem (see Theorem 2.2 in \cite{rosendal2009automatic}).

\subsubsection*{Support-preserving endomorphisms of $\cPr(\Z^d)$ and beyond} Our techniques do not extend beyond $d=1$ in a straightforward way, as multivariate polynomials do not generally decompose into a product of simple factors, such as the quadratic polynomials in the one-dimensional case. We thus offer the following question: Is there, for every support-preserving endomorphism $\Phi \colon \cPr(\Z^d) \to \cPr(\Z^d)$, a vector $\beta = (\beta_1,\ldots,\beta_d)$ such that 
\begin{align*}
    \Phi[\mu](x) = \frac{\mu(x) \ee^{-\beta \cdot x}}{\sum_y \mu(y)\ee^{-\beta \cdot y}}?
\end{align*}

More generally, given a semi-group $G$, does there always exist a homomorphism $\pi \colon G \to \R$ such that every support-preserving $\Phi \colon \cPr(G) \to \cPr(G)$ is of the form
\begin{align*}
    \Phi[\mu](x) = \frac{\mu(x) \ee^{-\pi(x)}}{\sum_y \mu(y)\ee^{-\pi(y)}}?
\end{align*}

\subsubsection*{Weakening the support-preserving requirement}

For $S\subset \R$, say that $\Psi \colon \cPr(S) \to \cPr(S)$ is \emph{weakly support-preserving} if $\Psi[\mu]$ is absolutely continuous with respect to $\mu$, i.e., the support of $\Psi[\mu]$ is a subset of the support of $\mu$. Clearly, every support-preserving $\Phi$ is also weakly support-preserving, and so the class of support-preserving endomorphisms of, say, $\cPr(\Z_{\geq 0})$ is contained in the weakly support-preserving ones. We conjecture that the set of all 
 weakly support-preserving endomorphisms is exhausted by $\Phi_
\beta$ and the two limiting cases $\Phi_{+\infty} = \lim_{\beta\to+\infty}\Phi_\beta$ and $\Phi_{-\infty}=\lim_{\beta\to-\infty}\Phi_\beta$, which correspond to  putting a point mass on the minimal or maximal point of the support, respectively.

\subsection*{Related literature}
This paper is related  to other work on  polynomials with non-negative coefficients. This literature consists of two lines of research. 

The line closest to our analysis originated from the classical works of Poincar\'e~\cite{poincare1883equations} and P\'olya~\cite{polya1928uber},  exploring the relation of non-negativity of a polynomial $p$ and the possibility of finding a factor $r$ such that $r\cdot p$ has non-negative coefficients under various assumptions on $r$; see a recent contribution by Michelen and Sahasrabudhe~\cite{michelen2019characterization} for a survey. 

A related strain of research explores the connection between the coefficients and the distribution of zeros. Since polynomials with non-negative coefficients are moment-generating functions of finitely-supported distributions, this direction is tightly related to non-classical limit theorems of probability theory; see the series of papers by Michelen and Sahasrabudhe  for recent progress~\cites{michelen2019central,michelen2019central2,michelen2021anti}. 

The algebra of the semi-group of probability measures on the reals under convolutions is well studied, including its homomorphisms to $\R$, $\Z$ and $\C$; see ``Algebraic Probability Theory,'' a book by  Ruzsa and Sz{\'e}kely~\cite{ruzsa1988algebraic}, as well as more recent work~\cites{mu2024monotone, fritz2019monotone}. Homomorphisms into general groups were considered by Mattner~\cite{mattner2004cumulants}.

\subsection*{Acknowledgements}
We thank Tim Austin, Alexander Guterman, Ramon van Handel, Tom Hutchcroft, Daniel Litt, Gil Refael, Barry Simon, and Stanislav Smirnov for illuminating conversations and helpful suggestions.

\section{Preliminaries}
\label{sec:prelim}
Recall that $\cP$ is the set of polynomials $p(x) = \sum_{k=0}^d p_k x^k$ such that $p_k \geq 0$ and $p(1)=\sum_k p_k = 1$. We denote by $\deg(p)$ the degree of $p$. The set $\cP$ is contained in $\cM$, the set of polynomials $p$ such that $p(1)=1$ and $p(x)>0$ for all $x>0$. Note that the latter condition can be equivalently changed to $p(x) \neq 0$ for all $x>0$, so that $\cM$ consists of the polynomials $p$ with no positive roots such that $p(1)=1$.

For notational convenience,  we will sometimes omit normalization constants, often writing these polynomials in monic form. For example, $p(x)=(x+1)/2 \in \cP$ will be written as $p(x)=x+1$. Similarly, an expression of the form $\Phi[(x+1)/2] = (x+2)/3$ will be written more succinctly as $\Phi[x+1]=x+2$. Since every polynomial with non-positive roots can be normalized to a unique $p \in \cM$, and since normalization preserves the support and commutes with multiplication, this will introduce no ambiguity.

For $\gamma>0$, let $\Psi_\gamma \colon \cM \to \cM$ be given by
\begin{align}
    \label{eq:psi}
    \Psi_\gamma \colon \cM &\to \cM\nonumber \\
    p(x) &\mapsto p(x / \gamma).
\end{align}
Note that we omit normalization constants, as explained above; the normalized form is $\Psi_\gamma[p] = p(x/\gamma)/p(1/\gamma)$. In terms of probability measures, $\Psi_\gamma$ corresponds to the map $\Phi_\beta$ from the introduction, for $\beta = \log \gamma$. In particular, it is easy to verify that $\Psi_\gamma$ is support-preserving and multiplicative, and---importantly---that it maps $\cP$ to $\cP$.

The following lemma is a strengthening of Lemma~\ref{lem:rs}. 
\begin{lemma}[B\'alinth]
\label{lem:polya}
For every $q \in \cM$ it holds for all $n$ large enough that $q(x)\cdot(x+1)^n \in \cP$.
\end{lemma}
Note that we here again drop the normalization constant and write $(x+1)^n$ rather than $2^{-n}(x+1)^n$. This result is a consequence of P\'olya's Positivstellensatz~\cite{polya1928uber}, a more general statement about multivariate polynomials; see also Theorem~56 in~\cite{hardy1952inequalities}. In the univariate setting, a version of this result was obtained by B\'alinth; see a footnote on the first page of~\cite{polya1928uber}.

We will need a slight strengthening of this lemma.
\begin{lemma}
\label{lem:polya-r}
For every $q \in \cM$ and all $\gamma > 0$ it holds for all $n$ large enough that $q(x)(x+\gamma)^n \in \cP$.
\end{lemma}
\begin{proof}
Note that $\Psi_\gamma$ maps  $x+1$ to $x+\gamma$ (omitting normalization). Fix $q \in \cM$. By Lemma~\ref{lem:polya}, we know that  $\Psi_{\gamma^{-1}}[q](x) \cdot (x+1)^{n}$ is in $\cP$ for all $n$ large enough. Since $\Psi_\gamma$ is multiplicative,
\begin{align*}
    \Psi_\gamma\big[\Psi_{\gamma^{-1}}[q](x) \cdot (x+1)^{n}\big] = q(x) \cdot \Psi_\gamma[(x+1)^n] = q(x)(x+\gamma)^n,
\end{align*}
and since $\Psi_\gamma$ maps $\cP$ to $\cP$, it follows that $q(x)(x+\gamma)^n$ is also in $\cP$ for all $n$ large enough.
\end{proof}

\section{Support-preserving multiplicative maps of polynomials with non-negative coefficients}

In this section, we prepare the ingredients for proofs of Theorems~\ref{thm:main},~\ref{thm:main-poly}, and~\ref{thm:main-poly-rational} and then prove the theorems. 
The first ingredient is extending a multiplicative support-preserving $\Phi$ from the set of polynomials with non-negative coefficients $\cP$ to the set $\cM$ of polynomials without positive roots. 

\subsection{Extending the domain of $\Phi$}  In this section we prove the following result.
\begin{proposition}\label{prop:extension}
  Every support-preserving multiplicative $\Phi \colon \cP \to \cP$ can be (uniquely) extended to a degree-preserving multiplicative $\Phi \colon \cM \to \cM$.
\end{proposition}

Let
\begin{align*}
    \F(\Z_{\geq 0}) = \left\{\frac{p'}{p}\,:\, p,p' \in \cP\right\}.
\end{align*}
The first step towards proving Proposition~\ref{prop:extension} is to extend a support-preserving multiplicative $\Phi \colon \cP \to \cP$ to a multiplicative $\Phi \colon \cM \to \F(\Z_{\geq 0})$. 

To this end, given $q \in \cM$, there is, by Lemma~\ref{lem:rs}, $p,r \in \cP$  such that $ q \cdot p = r$. Define the extension of $\Phi$ to $\cM$ by
  \begin{align}
  \label{eq:extension}
       \Phi[q] = \frac{\Phi[r]}{\Phi[p]}.
  \end{align}
  To see that this is well defined, suppose that $q \cdot p' =r'$ for some $p',r' \in \cP$, and note that $q \cdot p \cdot p' \in \cP$. Thus 
  \begin{align*}
    \Phi[r] \cdot \Phi[p'] = \Phi[q \cdot p] \cdot \Phi[p'] =\Phi[q \cdot p \cdot p'] =  \Phi[q \cdot p'] \cdot \Phi[p]=\Phi[r']\cdot \Phi[p],
  \end{align*}
  and so 
  \begin{align*}
      \frac{\Phi[r]}{\Phi[p]} =\frac{\Phi[r']}{\Phi[p']}.
  \end{align*}

  To see that $\Phi$ is multiplicative, i.e., that  $\Phi[q \cdot q'] = \Phi[q] \cdot \Phi[q']$ for all $q,q' \in \cM$, suppose that $q\cdot p$ and  $q'\cdot p'$ belong to $\cP$. Thus $q \cdot q' \cdot p \cdot p' \in \cP$. Hence,
  \begin{align*}
     \Phi[q \cdot q'] = \frac{\Phi[q \cdot q \cdot p \cdot p']}{\Phi[p \cdot p']} = \frac{\Phi[q \cdot p ]}{\Phi[p]}\frac{\Phi[q' \cdot p']}{\Phi[p']} = \Phi[q] \cdot \Phi[q'].
  \end{align*}
  We note that the extension defined by \eqref{eq:extension} is unique, since any extension that satisfies $\Phi[q \cdot q'] = \Phi[q] \cdot \Phi[q']$ must satisfy \eqref{eq:extension} whenever $q \cdot p = r$.

To prove Proposition~\ref{prop:extension}, we need to show that the image of this extension is, in fact, in the polynomials. To this end, for $0 < a < b$, let\footnote{As discussed in \S\ref{sec:prelim}, we omit a normalizing constant and write $q^{a,b}(x)$ as above, rather than $(x^2-ax+ab)/(1-a+ab)$.}
\begin{align*}
    q^{a,b}(x)=x^2-a x+a b.
\end{align*} 
Note that $q^{a,b} \in \cM$.

Then
\begin{align*}
    q^{a,b}(x)\cdot(x+t) = (x^2-a x+a b)\cdot (x+t) = x^3+(t-a)x^2+a(b-t)x+a b t,
\end{align*}
and in particular
\begin{align*}
    q^{a,b}(x)\cdot(x+a) &= x^3+a(b-a)x +a^2 b \\
    q^{a,b}(x)\cdot(x+b) &= x^3+(b-a)x^2 
 +a b^2.
\end{align*}
are both in $\cP$. Importantly, they have different supports. By  \eqref{eq:extension},
\begin{align*}
    \Phi[q^{a,b}] = \frac{\Phi[q^{a,b}(x)\cdot(x+a)]}{\Phi[x+a]} = \frac{\Phi[q^{a,b}(x)\cdot(x+b)]}{\Phi[x+b]}.
\end{align*}
Now, because $q^{a,b}(x)\cdot(x+a)$ and $ q^{a,b}(x)\cdot(x+b)$ have different supports, it follows from the support-preserving property of $\Phi$ that
\begin{align*}
    \Phi[q^{a,b}(x)\cdot(x+a)] \neq \Phi[q^{a,b}(x)\cdot(x+b)].
\end{align*}
Hence $\Phi[x+a] \neq \Phi[x+b]$. We have thus proved the following claim:\footnote{Again, we omit normalization constants as discussed in \S\ref{sec:prelim}. With these constants the claim would be that $a \neq b$ implies $\Phi[(x+a)/(1+a)] \neq \Phi[(x+b)/(1+b)]$.}
\begin{claim}
\label{clm:degree-1-unequal}
Suppose that $\Phi$ is support-preserving and multiplicative. Then  $\Phi[x+a] \neq \Phi[x+b]$ for $a \neq b$.
\end{claim}

With this claim, we are ready to prove our proposition.
\begin{proof}[Proof of Proposition~\ref{prop:extension}]
Let $\Phi \colon \cP \to \cP$ be support-preserving and multiplicative. Extend it to $\Phi \colon \cP \to \F(\Z_{\geq 0})$ using \eqref{eq:extension}. Choose any $a,b >0$, $a \neq b$, and fix $q \in \cM$. By Lemma~\ref{lem:polya-r}, there is $n$ large enough so that both $q(x)(x+a)^n$ and $q(x)(x+b)^n$ are in $\cP$. We thus have that
\begin{align*}
    \Phi[q] = \frac{\Phi[q(x)\cdot(x+a)^n]}{\Phi[(x+a)^n]} = \frac{\Phi[q(x)\cdot(x+b)^n]}{\Phi[(x+b)^n]}.
\end{align*}
By Claim~\ref{clm:degree-1-unequal} and the support-preserving property, we know that there are some $c \neq d$ such that $\Phi[x+a] = x+c$ and $\Phi[x+b] = x + d$. Hence,
\begin{align*}
    \frac{\Phi[q(x)\cdot(x+a)^n]}{(x+c)^n} = \frac{\Phi[q(x)\cdot(x+b)^n]}{(x+d)^n},
\end{align*}
and rearranging we get
\begin{align*}
    \Phi[q(x)\cdot(x+a)^n](x+d)^n = \Phi[q(x)\cdot(x+b)^n]\cdot (x+c)^n.
\end{align*}
Both sides are polynomials and have the same roots. Of these roots, at least $n$ are equal to $-c$, since the right-hand side includes the term $(x+c)^n$. The left-hand side thus also has at least $n$ roots which are equal to $-c$. Since $c \neq d$, $\Phi[q(x)\cdot(x+a)^n]$ has at least $n$ roots equal to $-c$. It follows that
\begin{align*}
    \Phi[q] = \frac{\Phi[q(x)\cdot(x+a)^n]}{(x+c)^n}
\end{align*}
is a polynomial. Furthermore, it is in $\cM$, since it does not have positive roots. Finally, by the support-preserving property of $\Phi$,
\begin{align*}
    \deg(\Phi[q(x) \cdot (x+a)^n]) =\deg(q(x) \cdot (x+a)^n) = \deg(q)+n.
\end{align*}
On the other hand, since $\Phi$ is a multiplicative,
\begin{align*}
    \deg(\Phi[q(x) \cdot (x+a)^n])
    = \deg(\Phi[q] \cdot \Phi[(x+a)^n]) = \deg(\Phi[q]) + n,
\end{align*}
and so $\deg(\Phi[q]) = \deg(q)$.
\end{proof}

\subsection{Linear polynomials and some quadratic polynomials}\label{sec_binomials}

In this section, we prove the following proposition.
\begin{proposition}
\label{prop:binomials}
Let $\Phi \colon \cP \to \cP$ be  support-preserving and multiplicative. Then there exists a constant $\gamma > 0$ such that $\Phi[x+a] = x+\gamma a$ for all $a>0$.
\end{proposition}
Let $\Phi \colon \cP \to \cP$ be  support-preserving and multiplicative. Consider $\varphi\colon \R_{>0} \to \R_{>0}$  given  by $\Phi[x+t]=x+\varphi(t)$. We will show that $\varphi(a)=\gamma a$, thus proving Proposition~\ref{prop:binomials}. As a byproduct, we also show that $\Phi[x^2- a x + ab] = x^2-\gamma a x+ \gamma^2 a b$ for all $0 < a < b$  (Corollary~\ref{cor:easy-quad}).

As above, let $q^{a,b}(x)=x^2-a x+a b$, for $0 < a < b$. Then 
\begin{align*}
    q^{a,b}(x)\cdot(x+t) = (x^2-a x+a b)\cdot (x+t) = x^3+(t-a)x^2+a(b-t)x+a b t
\end{align*}
is in $\cP$ for any $t \in [a,b]$. By Proposition~\ref{prop:extension},
\begin{align}
\label{eq:Tqa}
    \Phi[q^{a,b}] = \frac{\Phi[q^{a,b}(x) \cdot (x+t)]}{\Phi[x+t]}= \frac{\Phi[x^3+(t-a)x^2+a(b- t)x+a b t]}{x+\varphi(t)}
\end{align}
is a polynomial for all $0<a<b$ and $t \in [a,b]$. In particular, for $t=a$ we get that
\begin{align*}
    \Phi[q^{a,b}] = \frac{\Phi[x^3+a( b-a)x+a^2 b ]}{x+\varphi(a)}
\end{align*}
is a polynomial. Using the support-preserving property of $\Phi$, we can write the numerator as $\Phi(x^3+a( b-a)x+a^2b) = x^3+c x+d$, which must have $x+\varphi(a)$ as a factor. Factoring $x+\varphi(a)$ from this polynomial yields that 
\begin{align}
\label{eq:Tqa1}
    \Phi[q^{a,b}] = x^2-\varphi(a)x+\varphi(a)^2+c.
\end{align}
Similarly, substituting $t=b$ into \eqref{eq:Tqa} we get
\begin{align*}
    \Phi[q^{a,b}] = \frac{\Phi(x^3+(b-a)x^2+a b^2)}{x+\varphi(b)}.
\end{align*}
Writing the numerator as $\Phi[x^3+(b-a)x+a b^2] = x^3+e x^2+f$, and since $x+\varphi(b)$ is a factor of this polynomial, we get that 
\begin{align}
  \label{eq:Tqa2}
    \Phi[q^{a,b}] = x^2-(\varphi(b)-e)x+\varphi(b)^2-e \varphi(b).
\end{align}
Equating \eqref{eq:Tqa1}  and \eqref{eq:Tqa2}, we get that $\varphi(a)=\varphi(b)-e$. This then yields that 
\begin{align}
\label{eq:easy-quad}
    \Phi[q^{a,b}] = x^2-\varphi(a)x+\varphi(a)\varphi(b).
\end{align}
Now, choose $t \in (a,b)$. Then $q^{a,b}(x)\cdot(x+t)$ is a cubic polynomial with positive coefficients and, by the above,
\begin{align*}
    \Phi[q^{a,b}(x)\cdot(x+t)] 
    &= (x^2-\varphi(a)x+\varphi(a)\varphi(b))(x+\varphi(t))\\
    &= x^3+(\varphi(t)-\varphi(a))x^2+\cdots.
\end{align*}
By the support-preserving property of $\Phi$, the coefficient of $x^2$ must be positive, and so we have shown that $\varphi$ is strictly monotone increasing:
\begin{claim}
\label{clm:monotone}
If $0 < a < b$, then $\varphi(a) < \varphi(b)$.
\end{claim}

We are now ready to prove the main result of this section.
\begin{proof}[Proof of Proposition~\ref{prop:binomials}]
Note that
\begin{align*}
  q^{a,b}(x)\cdot(x+2b)^2
  &= (x^2-a x+a b)\cdot (x+2b)^2\\
  &= x^4 + (4 b-a) x^3+ b(4 b-3 a) x^2  + 4 a b^3.
\end{align*}
is in $\cP$. Since
\begin{align*}
  \Phi[q^{a,b}(x)]\cdot \Phi[(x+2b)^2]
  &= (x^2-\varphi(a)x+\varphi(a)\varphi(b))\cdot (x+\varphi(2b))^2\\
  &= \cdots + (2 \varphi(a) \varphi(b) \varphi(2b) - \varphi(a) \varphi(2b)^2) x + \cdots,
\end{align*}
it follows from the support-preserving property of $\Phi$ that
\begin{align*}
  2 \varphi(a) \varphi(b) \varphi(2b) - \varphi(a) \varphi(2b)^2 = 0
\end{align*}
or
\begin{align*}
    \varphi(2b) = 2 \varphi(b).
\end{align*}

Likewise,
\begin{align*}
  q^{a,b}(x)\cdot(x+3b)^3
  &= (x^2-a x+a b)\cdot (x+3b)^3\\
  &= x^5 + (9b-a)x^4 + b(27b-8a) x^3+ b^2(27b-18a) x^2  + 27 a b^4.
\end{align*}
is in $\cP$ for $b>a$. Since
\begin{align*}
  \Phi[q^{a,b}(x)]\cdot \Phi[(x+3b)^3]
  &= (x^2-\varphi(a)x+\varphi(a)\varphi(b))\cdot (x+\varphi(3b))^3\\
  &= \cdots + (3 \varphi(a) \varphi(b) \varphi(3b)^2 - \varphi(a) \varphi(3b)^3) x + \cdots,
\end{align*}
again applying the support-preserving property of $\Phi$ yields that
\begin{align*}
    3 \varphi(a) \varphi(b) \varphi(3b)^2 - \varphi(a) \varphi(3b)^3 = 0,
\end{align*}
or 
\begin{align*}
    \varphi(3b) = 3\varphi(b).
\end{align*}
It thus follows from Lemma~\ref{lem:func} below that there exists a constant $\gamma > 0$ such that $\varphi(b) = \gamma b$.
\end{proof}

\begin{lemma}
\label{lem:func}
Suppose that $f \colon \R_{>0} \to \R_{>0}$ is strictly monotone increasing, and satisfies $f(2x)=2f(x)$ and $f(3x)=3f(x)$. Then there exists a constant $c>0$ such that $f(x) = c x$.
\end{lemma}
\begin{proof}
Since 2 and 3 are coprime, the set
$X = \{2^m 3^n \,:\, m, n \in \Z\}$ is dense in $\R_{>0}$. Since $f(2x)=2f(x)$ and $f(3x)=3f(x)$, we obtain that $f(x) = x f(1)$ for all $x \in X$.  Given any $y \in \R_{>0}$, choose a sequence $(x_n^+)_n$ in $X$ that converges to $y$ from above, and likewise $(x_n^-)_n$ that converges to $y$ from below. Then, by the monotonicity of $f$,
\begin{align*}
  y f(1) = \lim_n x_n^+f(1) = \lim_n f(x_n^+) \geq f(y) \geq \lim_n f(x_n^-) = \lim_n x_n^-f(1) = y f(1).
\end{align*}
In particular, $f(y)=f(1) y$.
\end{proof}

We end this section with the following result which is a corollary of Proposition~\ref{prop:binomials} and identity \eqref{eq:easy-quad}:
\begin{corollary}
\label{cor:easy-quad}
Consider an extension of a support-preserving and multiplicative $\Phi \colon \cP \to \cP$ to $\cM$ and define $\gamma$ by  $\Phi[x+1]=x+\gamma$. Then $\Phi[x^2-a x+ b a] = x^2-\gamma a x + \gamma^2 a b$ for all  $0 < a < b$.
\end{corollary}

\subsection{Quadratic polynomials}

Consider a polynomial of the form $q^a(x)=x^2-a x+ 1$. For $q^a$ to be in $\cM$, the discriminant must be negative, i.e.,  $a<2$. 

\begin{claim}
\label{cl:invariance_small_interval0}
Suppose that $\Phi$ is support-preserving and multiplicative, and $\Phi[x+1]=x+1$. Then   $\Phi[q^a]=q^a$ for all $a \in [-1,1]\setminus \{0\}$.
\end{claim}
\begin{proof}
For $a \in (0,1)$, we know from Corollary~\ref{cor:easy-quad} that $\Phi[q^a]=q^a$, by choosing $b = 1/a$.  We next demonstrate that $\Phi[q^a]=q^a$ for $a=1$. Let $\Phi[x^2-x+1]=x^2-a'x+b'$. Note that
$$(x^2-x+1)(x+1)=x^3+1.$$
Hence, by the support-preserving property of $\Phi$, in the expression
$$\Phi[(x^2-x+1)(x+1)]=(x^2-a'x+b')(x+1)=x^3+(1-a')x^2+(b'-a')x+ b'$$
 the coefficients of $x$ and $x^2$ must vanish. We obtain $1-a'=0$ and $b'-a'=0$. Thus $a'=b'=1$ and so $\Phi[q^a]=q^a$ for $a=1$.

Now suppose $a \in [-1,0)$. Then $q^a \in \cP$ so $\Phi(q^a) = x^2+a' x + b'$ for some $a',b' >0$. Then
\begin{align*}
    \Phi[q^{-a} \cdot q^{a}]
    &= \Phi[q^{-a}] \cdot \Phi[q^{a}]\\
    &= (x^2-a x+ 1)\cdot(x^2+a' x+ b') \\
    &= x^4 +(a'-a)x^3+ (1+b'-a a')x^2+(a'-a b')x + b'.
\end{align*}
Now,
\begin{align*}
    q^a(x) \cdot q^{-a}(x) = (x^2-a x+ 1)\cdot(x^2+a x+ 1) = x^4 + (2-a^2)x^2+1,
\end{align*}
and so,  by the support-preserving property of $\Phi$, we have that $a'-a=0$ and $a'-a b'=0$. Hence, $a'=a$ and $b'=1$. We conclude that $\Phi[q^a]=q^a$ for $a \in [-1,1]\setminus\{0\}$. 
\end{proof}

Define $\pi \colon \cM \to \cM$ to be the map $q(x) \mapsto q(x^2)$. Let $\Phi^{\pi} = \pi^{-1} \circ \Phi \circ \pi$, so that
\begin{align*}
    \Phi^{\pi}[q_0+q_1 x + q_2 x^2 + \cdots + q_d x^d] = r_0 + r_1 x + r_2 x^2 + \cdots + r_d x^d
\end{align*}
whenever 
\begin{align*}
    \Phi[q_0+q_1 x^2 + q_2 x^4 + \cdots + q_d x^{2d}] = r_0 + r_1 x^2 + r_2 x^4 + \cdots + r_d x^{2d}.
\end{align*}
It is easy to verify that $\Phi^{\pi} \colon \cP \to \cP$ is 
support-preserving and multiplicative if $\Phi$ is. By Proposition~\ref{prop:extension}, $\Phi^{\pi}$ admits a unique extension to a degree-preserving multiplicative map $\cM \to \cM$. 

\begin{claim}\label{cl:invariance_small_interval}
Suppose that $\Phi$ is support-preserving and multiplicative, and $\Phi[x+1]=x+1$. Then for all $a \in [-1,1]$ it holds that $\Phi[q^a]=q^a$.
\end{claim}
\begin{proof}
The case of $a \neq 0$ was shown in Claim~\ref{cl:invariance_small_interval0}. It remains to be shown that $\Phi[x^2+1]=x^2+1$, or equivalently that $\Phi^\pi[x+1]=x+1$. Since $\Phi^\pi$ is  support-preserving and multiplicative, by Proposition~\ref{prop:binomials} there is a constant $\gamma>0$ such that $\Phi^\pi[x+a]=x+ \gamma a$. 

Define $\Phi'=\Psi_{1/\gamma}\circ \Phi^\pi$, where we recall that $\Psi_{1/\gamma}\colon p(x)\to p(\gamma x)$. Hence, $\Phi'$ is support-preserving and multiplicative and satisfies $\Phi'[x+1]=x+1$. Therefore, by Claim~\ref{cl:invariance_small_interval0}, $\Phi'[x^2-a+1]=x^2-ax+1$ for all  $a \in [-1,1] \setminus \{0\}$. Since $\Phi^\pi=\Psi_\gamma\circ \Phi'$, we get $\Phi^\pi[q^a(x)] = x^2-\gamma a x+\gamma^2$ for all  $a \in [-1,1] \setminus \{0\}$.

Let us show that $\gamma=1$. Observe that for $a=-1$
\begin{align*}
    q^{a}(x^2)= x^4+x^2+1 = (x^2+x+1)\cdot(x^2-x+1) = q^{a}(x)\cdot q^{-a}(x).
\end{align*}
Hence,
$$\Phi[q^a(x^2)]=\Phi[q^a(x)]\cdot \Phi[q^{-a}(x)]=(x^2+x+1)(x^2-x+1)=x^4+x^2+1.$$
On the other hand,
$$\Phi[q^a(x^2)]=\Phi^\pi[q^a](x^2)=x^4-\gamma a x^2+\gamma^2.$$
Thus $\gamma=1$, which completes the proof.
\end{proof}

\begin{claim}\label{cl:interval_extension}
Let $A$ be a subset of $(-2,2)$. 
Suppose that $\Phi^\pi[q^a]=q^a$ for all $a \in A$. Then $\Phi[q^a]=q^a$ for all $a$ such that $a^2-2 \in A$.
\end{claim}
\begin{proof}
Note that 
\begin{align*}
    q^a(x) \cdot q^{-a}(x) = (x^2-a x+ 1)\cdot(x^2+a x+ 1) = x^4 + (2-a^2)x^2+1.
\end{align*}
By the claim hypothesis
$\Phi^\pi[x^2+(2-a^2)x+1] = x^2+(2-a^2)x+1$. It follows that
$$\Phi[q^a]\cdot \Phi[q^{-a}]=x^4+(2-a^2)x^2+1. $$
Without loss of generality, we can assume that $a\geq 0$. By Proposition~\ref{prop:extension}, the left-hand side is the product of two quadratic polynomials; moreover,  $\Phi[q^{-a}]$ has non-negative coefficients.
The polynomial $x^4+(2-a^2)x^2+1$ on the right-hand side has two pairs of complex-conjugate roots. Hence, there is a unique way to represent it as a product of two quadratic monic polynomials with real coefficients:
$$x^4+(2-a^2)x^2+1=(x^2-a x+ 1)\cdot(x^2+a x+ 1).$$
Only one of these quadratic factors has non-negative coefficients and thus
$$\Phi[q^{a}]=x^2-a x+ 1\qquad\mbox{and}\qquad \Phi[q^{-a}]=x^2+a x+ 1$$
completing the proof.
\end{proof}
\begin{claim}
If a support-preserving multiplicative $\Phi$ satisfies $\Phi[x+1]=x+1$, then $\Phi[q^a]=q^a$ for all $a$ in $(-2,2)$.
\end{claim}
\begin{proof}
Let $A^*$ be the set of all $a \in (-2,2)$ such that $\Phi'[q^a]=q^a$ for all support-preserving multiplicative maps $\Phi'$ such that $\Phi'[x+1]=x+1$. By Claim~\ref{cl:invariance_small_interval}, $[-1,1]\subseteq A^*$.

Let $f(x) = x^2-2$, and denote by $f^{(n)}$ the $n$-fold composition of $f$ with itself. We claim that for any $x \in (1,2)$ there is a number $n$ such that $f^{(n)}(x) \in [-1,1]$. Since the image of $(1,2)$ under $f$ is $(-1,2)$, it is enough to show that there is no $x_0\in (1,2)$ such that $x_n=f^{(n)}(x_0)$ stays in $(1,2)$ for all $n$. Towards a contradiction, suppose that such $x_0$ exists. For $x \in (1,2)$, we have that $f(x) < 3x-4$, since $f(1)=-1$, $f(2)=2$ and $f$ is strictly convex. In particular, $f(x) -x < 2x-4 < 0$ for all $x \in [1,2)$, so that $f(x)<x$. Thus the sequence $x_n$ is decreasing. Denote $x_\infty=\lim_ n x_n\in [1,2)$. By continuity of $f$, we get $f(x_\infty)=x_\infty$. But $f(x)<x$ for all $x \in [1,2)$. This contradiction implies that for any $x\in (1,2)$, there is $n$ such that $f^{(n)}(x)\in (-1,1)$. The same argument applies to $x \in (-2,-1)$, since $f(-x)=f(x)$.

By Claim~\ref{cl:interval_extension}, 
if $f(a) \in A^*$ then $a \in A^*$. It follows that if $f^{(n)}(a) \in A^*$, then $a \in A^*$. Since $[-1,1] \subseteq A^*$, we conclude that $(-2,2) \subseteq A^*$. Thus $A^* = (-2,2)$.

\end{proof}

We are now ready to show that if $\Phi[x+1]=x+1$, then $\Phi[q]=q$ for all quadratic $q \in \cM$ that have no real roots. Note that any such $q$ is (up to normalization) of the form $q(x) = x^2-a\gamma x+\gamma^2$ for some $a \in (-2,2)$ and   $\gamma>0$.
\begin{proposition}
\label{cl:all_quadratic_without_real_roots}
If a support-preserving multiplicative $\Phi$ satisfies $\Phi[x+1]=x+1$, then
$\Phi[x^2-a\gamma x+\gamma^2]=x^2-a\gamma x+\gamma^2$ for all $a \in (-2,2)$ and all  $\gamma>0$.
\end{proposition}
\begin{proof}
Recall that $\Psi_\gamma$ maps $p(x)$ to $p(x/\gamma)$. Consider $\Phi'=\Psi_{1/\gamma}\circ\Phi\circ\Psi_\gamma $. Then $\Phi'$ is support-preserving and multiplicative, and $\Phi'[x+1]=x+1$. By Claim~\ref{cl:interval_extension},
$\Phi'[q^a]=q^a$ for all $a \in (-2,2)$. Equivalently, $\Psi_{1/\gamma}\circ  \Phi\circ \Psi_\gamma [q_a]=q_a$. Applying $\Psi_\gamma$ on both sides of this identity, we get $\Phi\big[\Psi_\gamma [q_a]\big]=\Psi_\gamma[q_a]$. Since $\Psi_\gamma[q_a](x)=x^2-a\gamma x+\gamma^2$, the proof is complete.
\end{proof}

\subsection{Proofs of Theorems~\ref{thm:main} and~\ref{thm:main-poly}}

\begin{proof}[Proof of Theorem~\ref{thm:main-poly}]

We first claim that to prove the theorem, it suffices to show that a support-preserving and multiplicative $\Phi$ such $\Phi[x+1]=x+1$ is the identity map. To see that this statement implies the theorem, let $\Phi$ be support-preserving and multiplicative. Define $\gamma>0$ by $\Phi[x+1]=x+\gamma$. Recall from \eqref{eq:psi} that $\Psi_\gamma \colon \cP \to \cP$ is the map that takes $p(x)$ to $p(x / \gamma)/p(1/\gamma)$. Hence $\Phi'= \Psi_\gamma^{-1} \circ \Phi$ is support-preserving and multiplicative, and  furthermore satisfies $\Phi'[x+1]=x+1$. Hence, if we show that $\Phi'$ is the identity map, then we have shown that $\Phi = \Psi_\gamma$.

We now show that a support-preserving and multiplicative $\Phi \colon \cP \to \cP$ such that $\Phi[x+1]=x+1$ is the identity map.

Fix $p \in \cP$. By the fundamental theorem of algebra, we can write it as a product of polynomials
\begin{align*}
    p(x) = \prod_i r^i(x)\prod_j q^j(x),
\end{align*}
where each $r^i$ is linear and each $q^j$ is quadratic with no real roots. 

Since $p \in \cP$ has no positive roots, each linear term $r^i$ is of the form $r^i(x)=x+b_i$ for some $b_i \geq 0$. Since each quadratic term has no real roots, it is of the form $q^j(x) = x^2-a_j\gamma_j x+\gamma_j^2$ for some $\gamma_j >0$ and $a_j \in (-2,-2)$.

By Proposition~\ref{prop:extension}, we can extend $\Phi$ to a multiplicative map $\Phi \colon \cM \to \cM$. Hence,
\begin{align*}
    \Phi[p] = \prod_i \Phi[r^i]\prod_j\Phi[q^j].
\end{align*}
By Proposition~\ref{prop:binomials}, $\Phi[r^i]=r^i$. And by Proposition~\ref{cl:all_quadratic_without_real_roots}, $\Phi[q^j] = q^j$. We conclude that $\Phi[p]=p$, and so $\Phi$ is the identity map. By the remark at the beginning of the proof, this implies the theorem statement.
\end{proof}

\begin{proof}[Proof of Theorem~\ref{thm:main}]
The case $S=\Z_{\geq 0}$ follows immediately from Theorem~\ref{thm:main-poly} by translating from probability-generating functions back to probability measures.

Consider now $S=\Z$, and let $\Phi \colon \cPr(\Z) \to \cPr(\Z)$ be a support-preserving endomorphism. Then its restriction to $\cPr(\Z_{\geq 0})$ is equal to some $\Phi_\beta$. Given $\mu \in \cPr(\Z)$, there is some $z \in \Z$ such that $\mu * \delta_z \in \cPr(\Z_{\geq 0})$. Hence
\begin{align*}
    \Phi[\mu*\delta_z] = \Phi_\beta[\mu * \delta_z] = \Phi_\beta[\mu] * \Phi_\beta[\delta_z] = \Phi_\beta[\mu] * \delta_z,
\end{align*}
since $\Phi_\beta$ is a support-preserving endomorphism. On the other hand, 
\begin{align*}
    \Phi[\mu*\delta_z] = \Phi[\mu] * \Phi[\delta_z] = \Phi[\mu] * \delta_z,
\end{align*}
since $\Phi$ is a support-preserving endomorphism. Hence
\begin{align*}
    \Phi[\mu] * \delta_z = \Phi_\beta[\mu] * \delta_z
\end{align*}
and so $\Phi[\mu] = \Phi_\beta[\mu]$.

Finally, consider the case $S=\Q$, and let $\Phi\colon \cPr(\Q) \to \cPr(\Q)$ be a support-preserving endomorphism. For each $n \in \Z_{> 0}$, the semi-group $\cPr(\Z/n)$ is isomorphic to $\cPr(\Z)$, and thus there is some $\beta_n$ such that the restriction of $\Phi$ to $\cPr(\Z/n)$ is equal to $\Phi_{\beta_n}$. But since $\Z/n$ and $\Z/m$ are both contained in $\Z/(nm)$, $\beta_n=\beta_m=\beta$. Finally, $\cPr(\Q) = \cup_n \cPr(\Z/n)$, and so $\Phi=\Phi_\beta$.
\end{proof}

\subsection{Proof of Theorem~\ref{thm:main-poly-rational}}

The proof of Theorem~\ref{thm:main-poly-rational} initially follows the argument of the proof of Theorem~\ref{thm:main-poly}. We first analogously extend $\Phi$ to a map from $\cM_\Q$ to $\cM$ where
\begin{align*}
    \cM_\Q = \left\{p(x) = \sum_{k=1}^d p_k x^k\,\middle\vert\, p_k\in \Q,  \  \ p(x)>0  \text{ for all } x>0,\ \ p(1)=1\right\}.
\end{align*}
The same argument as in the proof of Proposition~\ref{prop:extension} shows that there exists a multiplicative extension that preserves the degree. Note that instead of Poincar{\'e}'s Lemma (Lemma~\ref{lem:rs}), one can use P\'olya's Lemma (Lemma~\ref{lem:polya}) to ensure that for every $p \in \cM_\Q$ there exists an $r \in \cP_\Q$ such that $p \cdot r \in \cP_\Q$.

As in the proof of Theorem~\ref{thm:main-poly}, we first consider linear polynomials with rational coefficients and note that the same argument of Proposition~\ref{prop:binomials} shows that the analogous statement holds in the rational setting.
\begin{proposition}
\label{prop:binomials-rational}
For any support-preserving multiplicative $\Phi \colon \cP_\Q \to \cP$,  there exists a constant $\gamma > 0$ such that $\Phi[x+a] = x+\gamma a$ for all $a \in \Q_{>0}$.
\end{proposition}

We next study quadratic polynomials with rational coefficients. The argument of  Proposition~\ref{cl:all_quadratic_without_real_roots} still applies.
\begin{proposition}
\label{cl:all_quadratic_without_real_roots_rational}
Let  $\Phi \colon \cP_\Q \to \cP$ be support-preserving and multiplicative, and $\Phi[x+1]=x+1$. Then its extension to $\cM_\Q$ satisfies $\Phi[p]=p$ for any polynomial $p$ of the form $p(x)=x^2-a\gamma x+\gamma^2$ with $a \in (-2,2)\cap \Q$ and  $\gamma \in \Q_{>0}$.
\end{proposition}
Note that this does not apply to all rational quadratic polynomials, since the free coefficient is a square of a rational. 

Accordingly, let $\cM'_\Q \subset \cM_\Q$ be the set of polynomials with rational coefficients which are products of linear rational polynomials ($x+a$ for $a \in \Q_{\geq 0}$) and quadratic polynomials of the form considered in Proposition~\ref{prop:binomials-rational} ($x^2-a\gamma x+\gamma^2$ for $a \in (-2,2)\cap \Q$ and $\gamma \in \Q_{>0}$). Then the same proof of Theorem~\ref{thm:main-poly} yields the following:
\begin{proposition}
\label{prop:cpprime}
Let $\Phi$ be an extension of a support-preserving multiplicative map from $\cP_\Q$ to $\cM_\Q$. Then   there exists a constant $\gamma > 0$ such that $\Phi[p](x) = p(\gamma x)/p(\gamma)$ for any $p\in \cM'_\Q$.
\end{proposition}

Consider the topology on $\cM$ given by
$\lim_t p^{(t)} = p$ if $\lim_t \deg(p^{(t)}) = \deg(p)$ and $\lim_t p^{(t)}_k = p_k$ for all $k$. Then  $\cM'_\Q$ is a rich dense sub-semi-group of $\cM_\Q$, which is dense in $\cM$. Thus, automatic continuity (Proposition~\ref{prop:auto-cont}) together with Proposition~\ref{prop:cpprime} yield  Theorem~\ref{thm:main-poly-rational}.
\begin{proof}[Proof of Theorem~\ref{thm:main-poly-rational}]

Let $\Phi \colon \cP_\Q \to \cP$ be support-preserving and multiplicative. Define $\gamma$ by $\Phi[x+1]=x+\gamma$, and let $\Phi' = \Psi_\gamma^{-1} \circ \Phi$. Then $\Phi'$ is support-preserving, multiplicative, and satisfies $\Phi'[x+1]=x+1$.  Extend it to a multiplicative, degree-preserving $\Phi' \colon \cM_\Q \to \cM$. By Proposition~\ref{prop:cpprime}, the restriction of $\Phi'$  to $\cM'_\Q$ is of the form $\Psi_{\gamma'}$, and since $\Phi'[x+1]=x+1$, we get $\gamma'=1$. Therefore, $\Phi'$ restricted to $\cM'_\Q$ is the identity map. Hence, by Proposition~\ref{prop:auto-cont}, $\Phi'$ is the identity map on the whole domain $\cM_\Q$. Thus $\Phi = \Psi_\gamma$.
\end{proof}

\subsection{Proof of Proposition~\ref{prop:auto-cont}}

Recall that Poincar{\'e}'s Lemma (Lemma~\ref{lem:rs}) shows that for every $q \in \cM$ there is a polynomial $p \in \cP$ such that $q \cdot p \in \cP$. Towards proving Proposition~\ref{prop:auto-cont}, we investigate a related question: Given $p \in \cP$, for which $q \in \cM$ does it holds that $q \cdot p \in \cP$?

Given $p \in \cP$, denote by $S_p \subseteq \cM$ the set of polynomials $q$ such that $q \cdot p \in \cP$:
\begin{align*}
 S_p = \{q \in \cM\,:\, q \cdot p \in \cP\}.
\end{align*}
A natural question is to understand when $S_p \subseteq S_{p'}$. For example, if $p'(x) = p^2(x)$ then clearly $S_p \subseteq S_{p'}$. The next lemma shows that if we further require that $\deg(p') \leq \deg(p)$ (and $p_0 \neq 0$), then the containment $S_p \subseteq S_{p'}$ is only possible when $p=p'$.
\begin{lemma}\label{lm:cone}
Consider $p,p'\in \cP$ such that $\deg (p')\leq\deg (p)$ and  $p_0\ne 0$. 
Then $p = p'$ if and only if $S_p \subseteq S_{p'}$.
\end{lemma}
\begin{proof}
Only one direction is non-trivial: proving the existence of $q \in S_p \setminus S_{p'}$ if $p\ne p'$. 

Denote $n=\deg(p)$. We claim that we can assume without loss of generality that $p_k>0$ for $k=0,\ldots, n$. Otherwise, replace $p$ and $p'$ by $p\cdot r$ and $p'\cdot r$ with $r(x)=(x+1)^n$. If we can find $q$ such that $q \cdot p \cdot r \in \cP$ but $q \cdot p' \cdot r \not \in \cP$ then $q \cdot r \in S_p \setminus S_{p'}$.

The first step is to show the existence of a polynomial $s(x) = \sum_{k=0}^{n+1}s_k x^k$ with $\deg(s)=n+1$ such that all the coefficients $(p\cdot s)_{k}$  are non-negative for $k=0,\ldots, n+1$ but $(p'\cdot s)_{n+1}\leq -1$. These requirements are equivalent to a system  of linear inequalities on coefficients of $s$:
\begin{equation}\label{eq_linear_system}
A \vec{s} \geq \vec{b},\qquad \mbox{where}
\end{equation}
$$A=\begin{pmatrix}
p_{0} &     0         &  \cdots                 &     \cdots  &   0  \\
p_{1} & p_{0}          &  0                     &   \ddots   &   0  \\
\vdots &    \vdots      &     \ddots            &     \ddots &  0  \\
p_n    &  p_{n-1}       &  \cdots               & p_{0}    &   0      \\
0      & p_n            &  p_{n-1}              & \cdots    & p_{0}         \\
0      & -p'_n          &  -p'_{n-1}           & \cdots    & -p'_{0}  
\end{pmatrix}
,\qquad 
\vec{s}=\begin{pmatrix}
s_0\\
s_1\\
\vdots\\
s_{n+1}
\end{pmatrix}
,\qquad 
\vec{b}=\begin{pmatrix}
0\\
0\\
\vdots\\
\vdots\\
0\\
1
\end{pmatrix}.
$$
By the Farkas Lemma, \eqref{eq_linear_system} has a solution if and only if there is no way to combine the inequalities with non-negative coefficients to get the contradictory inequality $0\geq 1$. Formally, there must be no $(\lambda_0,\ldots,\lambda_{n+2}) \in \R_{\geq 0}^{n+2}$ such that 
\begin{equation}\label{eq_dual}
\left\{
\begin{array}{cc}
A^T \vec{\lambda}&=\vec{0}\\
\langle\vec{b},\vec{\lambda} \rangle&=1.
\end{array}
\right.
\end{equation}
Here $\lambda_l\geq 0$ is interpreted as the weight of the inequality~$l$ in the original system, and  $A^T$ denotes the transposed matrix $A$. Let us write down the equations~\eqref{eq_dual} explicitly:
$$
\left\{\begin{array}{rccr}
p_0\lambda_0+\ldots +p_n \lambda_n &=&0 &\\
\left(p_0\lambda_k+\ldots+p_{n-k}\lambda_{n}\right)+p_{n-k+1}\lambda_{n+1}-p'_{n-k+1}\lambda_{n+2}&=& 0, &k=1,\ldots, n+1\\ 
\lambda_{n+2}&=&1&
\end{array}
\right.
$$
Since all coefficients of $p$ are assumed to be positive and $\lambda_k\geq 0$ for all $k$, the first equation implies $\lambda_0=\lambda_1=\ldots=\lambda_n=0$. Hence, the second family of equations gives $p_{n-k+1}\lambda_{n+1}-p'_{n-k+1}\lambda_{n+2}=0$ for $k=1,\ldots, n+1$. Since $p$ and $p'$ are non-zero and not equal, these identities can only hold if $\lambda_{n+1}=\lambda_{n+2}=0$. Since $\lambda_{n+2}=1$ by the last equation, we conclude that the system~\eqref{eq_dual} does not have a solution. Thus the system~\eqref{eq_linear_system} has a solution and the polynomial $s(x)$ exists.

We now argue that the polynomial $q \in \cM$ defined (up to normalization) by 
$$q(x)=s(x)+C\cdot x^{n+2}$$
 has the desired properties for a large enough constant $C>0$. Indeed, the coefficients of $x^k$ with $k=0,\ldots, n+1$ in the products $p\cdot q$ and $p'\cdot q$ are the same as in $p\cdot s$ and $p'\cdot s$. Hence, $(p'\cdot q)_{n+1}\leq -1$ and so $p'\cdot q\notin \cP$. It remains to be shown that we can always choose $C$ so that $(p\cdot q)_k\geq 0$ for $k=n+2,\ldots 2n+2$. For such $k$, we get 
$$(p\cdot q)_k=(p\cdot s)_k+ C\cdot p_{k-(n+2)}$$
Since all the coefficients of $p$ are strictly positive, choosing  
$$C=\max_{k=n+2,\ldots 2n+2}\frac{(p\cdot s)_k}{p_{k-(n+2)}}$$
ensures that $p\cdot q \in \cP$. As $p\cdot q \in \cP$ and $p\in \cP$, we conclude that the constructed polynomial $q$ belongs to $\cM$ and complete the proof.
\end{proof}

The next lemma strengthens the previous, again considering $p,p'$ such that $\deg(p') \leq \deg(p)$ and $p_0 \neq 0$. It shows that when $p \neq p'$, then not only is $S_p$ not contained in $S_{p'}$, but moreover, the interior of $S_p$ is not contained in $S_{p'}$. Recall that our topology is the one under which  $\lim_t p^{(t)} = p$ if $\lim_t \deg(p^{(t)}) = \deg(p)$ and $\lim_t p^{(t)}_k = p_k$ for all $k$. Note that $S_p$ trivially has a non-empty interior because, for any $p$, it contains $\cP$, which has a non-empty interior.  The lemma implies that the interior of $S_p$ extends non-trivially beyond that.\begin{lemma}\label{lm:cone_interior}
Consider $p, p'\in \cP$ such that $\deg (p')\leq\deg (p)$ and  $p_0\ne 0$. Then $p = p'$ if and only if the interior of $S_p$ is contained in $S_{p'}$.
\end{lemma}
\begin{proof}
One direction is again immediate. For the other direction, we need to show that if $M$ is a dense subset of $\cM$, then $p \neq p'$ implies that there is an element of~$M$ in  $S_p \setminus S_{p'}$.

Fix $p \neq p'$ with $p_0 \neq 0$ and $\deg(p') \leq \deg(p)$. By Lemma~\ref{lm:cone}, there is $q\in \cM$ such that $q \in S_p \setminus S_{p'}$, i.e., $p\cdot q$ has non-negative coefficients, but $p'\cdot q$ has a negative coefficient.  

Denote by $m$ the degree of $q$ and consider $r(x)=(x+1)^m$. For $\varepsilon>0$, define $q^\varepsilon=q+\varepsilon \cdot r$. Note that $p\cdot q^\varepsilon$ has strictly positive coefficients. Since the coefficients of a product depend continuously on the multiplicands, we can choose small $\varepsilon>0$ such that $p'\cdot q^\varepsilon$ has a negative coefficient. 

Let $q^{(t)}\in M$, $t=1,2,\ldots$, be a sequence of polynomials of degree $m$ such that 
$q^{(t)}_k\to q_k^\varepsilon$ as $t\to\infty$ for all $k=0,1,\ldots, m$.
Using the continuity of the product again, we conclude that, for large enough $t$, all the coefficients of $p\cdot q^{(t)}$ are strictly positive, but $p'\cdot q^{(t)}$ has a negative coefficient.
\end{proof}

\begin{proof}[Proof of Proposition~\ref{prop:auto-cont}]

We first consider a polynomial  $p\in M'\cap \cP$, prove that $\Phi[p]=p$, and then extend the conclusion to all $p\in M'$. 

We assume without loss of generality that $p_0 \neq 0$. Indeed, if $p(x) = x^m \cdot r(x)$ for some $r$ with $r_0 \neq 0$, then, by the richness hypothesis, $x$ is contained in $M$ and $r$ is contained in $M'$. Hence,  $\Phi[p] = \Phi[x]^m \cdot \Phi[r] = x^m \cdot \Phi[r]$ and proving $\Phi[r]=r$ would imply $\Phi[p]=p$. 

Towards a contradiction, assume that 
$p\ne \Phi[p]$ and denote $p'=\Phi[p]$. As $\Phi$ is degree-preserving, $\deg(p')=\deg(p)$. By Lemma~\ref{lm:cone_interior}, there is a polynomial $q\in M$ such that $q \in S_p$ and $q \not \in S_{p'}$. In other words, $p\cdot q$ has non-negative coefficients, and  $p'\cdot q$ has a negative coefficient. 

Since $q \in M$, by the proposition hypothesis, $\Phi[q]=q$ and thus
\begin{align*}
p'\cdot q= \Phi[p]\cdot\Phi[q]=\Phi[p\cdot q].
\end{align*}
Since $p\cdot q$ has non-negative coefficients, so does $\Phi[p\cdot q]$ because $\Phi$ maps $M'\cap \cP$ to $\cP$. We get that $p'\cdot q$ has no negative coefficients and
reach a contradiction. Thus $p'=p$ or, equivalently, $\Phi[p]=p$ for $p\in M'\cap \cP$.

It remains to show that $\Phi[p]=p$ for $p\in M'\setminus \cP$. By the density of $M$, there is a constant $a>0$ such that $q(x)=x+a$ belongs to $M$. By P\'olya's Lemma (Lemma~\ref{lem:polya-r}), $p(x)\cdot (x+a)^n$ has non-negative coefficients for some $n$ large enough. Denote $q(x)=(x+a)^n$. Both multiplicands in $p\cdot q$ belong to $M'$, so the product also does. Hence, $p\cdot q\in M'\cap \cP$. By the already proved statement, $\Phi[p\cdot q]=p\cdot q$. On the other hand, $\Phi[p\cdot q]=\Phi[p]\cdot \Phi[q]$. Since $q\in M$, we get $\Phi[q]=q$ and conclude that $\Phi[p]=p$.
\end{proof}

\bibliography{refs}

\end{document}